\documentclass[12pt,a4paper]{amsart}

\usepackage[utf8]{inputenc}
\usepackage[T1]{fontenc}
\usepackage[english]{babel}
\usepackage{cite}

\usepackage{indentfirst}
\usepackage{amssymb}
\usepackage{amsfonts}
\usepackage{amsmath}
\usepackage{amsthm}
\usepackage[top=2.5cm, bottom=2.5cm, left=2.5cm, right=2.5cm]{geometry}
\usepackage{amsopn}
\usepackage[colorlinks]{hyperref}
\usepackage{mathrsfs}
\usepackage{amsxtra}
\usepackage{color}
\usepackage{dsfont}

\newcommand{\al}{\alpha}
\newcommand{\be}{\beta}

\theoremstyle{plain}
\newtheorem{thm}{Theorem}

\newtheorem{lem}[thm]{Lemma}
\newtheorem{prop}[thm]{Proposition}

\theoremstyle{definition}
\newtheorem{defn}[thm]{Definition}

\newtheorem*{example*}{Example}

\newtheorem*{rem*}{Remark}
\newtheorem{rem}[thm]{Remark}

\newcommand{\R}{\mathbb{R}}

\DeclareMathOperator{\diam}{diam}
\DeclareMathOperator{\Sh}{\boldsymbol{\textbf{Sh}}}
\DeclareMathOperator{\SH}{\boldsymbol{\textbf{SH}}}

\DeclareMathOperator{\ucodim}{\underline{co\,dim}}
\DeclareMathOperator{\ocodim}{\overline{co\,dim}}
\DeclareMathOperator{\supp}{supp}
\DeclareMathOperator{\dist}{dist}
\DeclareMathOperator{\Int}{Int}

\DeclareSymbolFont{bbsymbol}{U}{bbold}{m}{n}
\DeclareMathSymbol{\ind}{\mathbin}{bbsymbol}{'061}

\title[Fractional Sobolev spaces with power weights]{Fractional Sobolev spaces with power weights}
\author[M{.} Kijaczko]{Micha\l{} Kijaczko}


\address[ M.K.]{Faculty of Pure and Applied Mathematics\\ Wroc{\l}aw University 
	of Science and Technology\\
	Wybrze\.ze Wyspia\'nskiego 27,
	50-370 Wroc{\l}aw, Poland
}
\email{michal.kijaczko@pwr.edu.pl}

\begin{document}

\begin{abstract}
We investigate the form of the closure of the smooth, compactly supported functions $C_{c}^{\infty}(\Omega)$ in the weighted fractional Sobolev space $W^{s,p;\,w,v}(\Omega)$ for bounded $\Omega$. We focus on the weights $w,\,v$ being powers of the distance to the boundary of the domain. Our results depend on the lower and upper Assouad codimension of the boundary of $\Omega$. For such weights we also prove the comparability between the full weighted fractional Gagliardo seminorm and the truncated one.
\end{abstract}

	\maketitle
	
	\section{Introduction and preliminaries}
		Let $\Omega\subset\mathbb{R}^{d}$ be an open set. Let $0<s<1$ and $1\leq p<\infty$. We recall that the \emph{fractional Sobolev space} is defined as 
	$$
	W^{s,p}(\Omega)=\left\{f\in L^{p}(\Omega):\int_{\Omega}\int_{\Omega}\frac{|f(x)-f(y)|^{p}}{|x-y|^{d+sp}}\,dy\,dx<\infty\right\}.
	$$
	This is a Banach space endowed with the norm 
	$$
	\|f\|_{W^{s,p}(\Omega)}=\|f\|_{L^{p}(\Omega)}+[f]_{W^{s,p}(\Omega)},
	$$
	where $[f]_{W^{s,p}(\Omega)}=\left(\int_{\Omega}\int_{\Omega}\frac{|f(x)-f(y)|^{p}}{|x-y|^{d+sp}}\,dy\,dx\right)^{1/p}$ is called the \emph{Gagliardo seminorm}.
	
	In this paper we consider weighted fractional Sobolev spaces. For \emph{weights} $w,v$ (i.e. measurable nonnegative functions on $\Omega$) we define the \emph{weighted Gagliardo seminorm} as
	$$
	[f]_{W^{s,p;\,w,v}(\Omega)}=\left(\int_{\Omega}\int_{\Omega}\frac{|f(x)-f(y)|^{p}}{|x-y|^{d+sp}}w(y)v(x)\,dy\,dx\right)^{\frac{1}{p}}
	$$
	and the \emph{weighted fractional Sobolev space} as 
	$$
	W^{s,p;\,w,v}(\Omega)=\left\{f\in L^{p}(\Omega):[f]_{W^{s,p;\,w,v}(\Omega)}<\infty\right\}.
	$$
	For bounded $\Omega$ the space defined above is always nonempty, because it contains constant functions. Moreover, if $w_{\alpha}(x)=\dist(x,\partial\Omega)^{-\alpha}$ and $v_{\beta}(y)=\dist(y,\partial\Omega)^{-\beta}$ for $\alpha,\beta\in\R$, we denote
	$$
	{W^{s,p;\,w_{\alpha},v_{\beta}}(\Omega)}=:{W^{s,p;\,\alpha,\beta}(\Omega)}.
	$$
	The space $W^{s,p;\,w,v}(\Omega)$ is equipped with the natural norm
	$$
	\|f\|_{W^{s,p;\,w,v}(\Omega)}=\|f\|_{L^{p}(\Omega)}+[f]_{W^{s,p;\,w,v}(\Omega)}.
	$$
	
	We remark here that all results of the paper remain true if we replace the space $L^{p}(\Omega)$ appearing in the definition of $W^{s,p;\,w,v}(\Omega)$ by the weighted analogue $L^{p}(\Omega,W)$ for any almost everywhere positive weight $W$, which is locally comparable to a constant (see Definition \ref{loc}) or continuous and satisfies $\int_{\Omega}W(x)\,dx<\infty$. Notice that the last condition ensures that the constant function $\ind_{\Omega}$ is in $L^{p}(\Omega,W)$. However, for simplicity we consider only the unweighted case.
	
	For an open set $\Omega$ we use the notation $d_{\Omega}(x)=\dist(x,\partial\Omega)$.
	
	\begin{defn}
	    By $W^{s,p;\,w,v}_{0}(\Omega)$ we denote the closure of $C_{c}^{\infty}(\Omega)\cap W^{s,p;\,w,v}(\Omega)$ (smooth functions with compact support in $\Omega$) in $W^{s,p;\,w,v}(\Omega)$ with respect to the weighted fractional Sobolev norm and by $W^{s,p;\,w,v}_{c}(\Omega)$ we denote the closure of all compactly supported, measurable functions in $\Omega$ (not necessarily smooth) in $W^{s,p;\,w,v}(\Omega)$ with respect to the weighted fractional Sobolev norm. We also denote $	W^{s,p;\,w_{\alpha},v_{\beta}}_{0}(\Omega)=:W^{s,p;\,\alpha,\beta}_{0}(\Omega),$ $	W^{s,p;\,w_{\alpha},v_{\beta}}_{c}(\Omega)=:W^{s,p;\,\alpha,\beta}_{c}(\Omega).$
	
	\end{defn}

We refer to Section \ref{section3} for a discussion on the cases when $C_c^{\infty}(\Omega)$ is or is not a subset of $W^{s,p;\,\alpha,\beta}(\Omega)$. In general, it may occure that the space $W^{s,p;\,\al,\be}_{0}(\Omega)$ is empty.
	   
	The main result of this paper is a generalization of the density result for unweighted fractional Sobolev spaces, which can be found in \cite[Theorem 2]{DK}. We present some necessary and sufficient conditions for the space $C_{c}^{\infty}(\Omega)$ to be dense in $W^{s,p;\,\alpha,\beta}(\Omega)$. In the negative case, under some additional assumptions we also find explicitly the form of the space $W_{0}^{s,p;\,\alpha,\beta}(\Omega)$. The necessary geometrical and technical definitions are contained in Section \ref{section2}. In Section \ref{section3} we present Lemmas, most of them being generalization of these from \cite{DK} and \cite{MR4190640} for the weighted case.

Let us remark that the weighted fractional Sobolev spaces related to the weighted Sobolev-type norm $[\,\cdot\,]_{W^{s,p;\,\alpha,\beta}(\Omega)}+\|\cdot\|_{L^{p^{*}}(\Omega,W)}$ and the problem of density of $C_{c}^{\infty}(\Omega)$ were investigated before by Dipierro and Valdinoci in \cite{MR3420496} for the case $\Omega=\R^d\setminus\{0\}$, $\alpha=\beta\in[0,(d-sp)/2)$, $p^{*}=dp/(d-sp)$ and $W(x)=|x|^{-\frac{2\alpha d}{d-sp}}$. However, this problem is not directly comparable to ours, because we consider only bounded sets $\Omega$. Similar weighted fractional Sobolev spaces were an object of study in \cite{MR3626031} in connection with weighted Caffarelli--Kohn--Nirenberg and fractional Hardy inequalities. Moreover, related results for unweighted Sobolev-type spaces can be found for example in \cite{MR3310082}, where the authors considered spaces of functions vanishing on $\R^d\setminus\Omega$ and in \cite{MR3989177}, where the problem of density of $C_c^{\infty}(\Omega)$ functions was investigated in the context of the fractional Sobolev spaces with variable exponents.
	
	Section \ref{section4} is devoted to the comparability result between the full weighted seminorm and the truncated one in the space $W^{s,p;\,w,v}(\Omega)$. This comparability is important to us in proving our main results (to be more specific - in Lemma \ref{inlcusion} and Lemma \ref{ineq}). However, it is also very interesting and nontrivial property itself. Similar results were obtained before by Dyda \cite{MR2215170} for Gagliardo-type seminorms with the additional homogeneous kernels (like indicators of cones), by Prats and Saksman \cite{MR3667439} in a more general context of Triebel-Lizorkin spaces and generalized later by Rutkowski \cite{MR4249694} for the kernels of the form 
	$|x-y|^{-d}\varphi(|x-y|)^{-q}$, with $\varphi$ satisfying certain technical assumptions. Some versions of the reduction of the integration theorems can also be found in \cite{MR3960670}, \cite{MR4111815} and \cite{MW}. We want to point here that a variant of comparability is nonexplicitly contained in the early work of Seeger \cite{MR1097208}. We prove a weighted analogue of the reduction of the integration theorem for the space $W^{s,p;\,\alpha,\beta}(\Omega)$, provided that $0\leq\alpha,\beta<\ucodim_{A}(\partial\Omega)$. This result is stated below.
	 \begin{thm}\label{tw.comp}
    Let $\Omega$ be a nonempty, bounded, uniform domain, let $0<s<1$ and $1\leq p<\infty$. Moreover, let $0<\theta\leq 1$. Suppose that $0\leq\al,\be<\ucodim_A(\partial(\Omega)$. Then the full seminorm $[f]_{W^{s,p;\,\al, \,\be}(\Omega)}$ and the truncated seminorm $$
    \left(\int_{\Omega}\int_{B\left(x,\theta d_{\Omega}(x)\right)}\frac{|f(x)-f(y)|^{p}}{|x-y|^{d+sp}}d_\Omega(y)^{-\be}d_\Omega(x)^{-\al}\,dy\,dx\right)^{\frac{1}{p}}
    $$
    are comparable, that is there exists a constant $C=C(\theta,d,s,p,\al,\be,\Omega)>0$ such that 
    $$
    \int_{\Omega}\int_{\Omega}\frac{|f(x)-f(y)|^{p}}{|x-y|^{d+sp}}\frac{dy}{d_\Omega(y)^{\be}}\frac{dx}{d_\Omega(x)^{\al}}\leq C\int_{\Omega}\int_{B\left(x,\theta d_{\Omega}(x)\right)}\frac{|f(x)-f(y)|^{p}}{|x-y|^{d+sp}}\frac{dy}{d_\Omega(y)^{\be}}\frac{dx}{d_\Omega(x)^{\al}}\,dy\,dx,
    $$
    for all $f\in L^{1}_{loc}(\Omega)$.
       \end{thm}
     
It is clear that the reverse inequality is trivial with constant equal to one, hence we indeed obtain the comparability between the full and the truncated weighted Gagliardo seminorms. Moreover, when $p=1$, the comparability can be formulated in a more general setting, for all $A_1$ class Muckenhoupt weights, see Theorem \ref{tw.comp2}.

Section \ref{sectionmainresults} contains proofs of our main results, Theorems \ref{tw1} and \ref{tw2}. Theorem \ref{tw1} is a generalization of \cite[Theorem 2]{DK} and Theorem \ref{tw2} is a generalization of \cite[Theorem 3]{DK}, provided that $\Omega$ is a uniform domain. 

\begin{thm}\label{tw1} 
Let $\Omega\subset \R^d$ be a nonempty, bounded, open set, let $0<s<1$, $1\leq p <\infty$ and $\al,\be\geq 0$.\\
(I) If  $sp+\alpha+\beta<d-\overline{\dim}_M(\partial\Omega)$,
then $W_{0}^{s,p;\,\alpha,\beta}(\Omega)=W^{s,p;\,\alpha,\beta}(\Omega)$.\\
(II)  If $\Omega$ is $(d-sp-\alpha-\beta)$-homogeneous, $p>1$ and
$sp+\alpha+\beta=\ucodim_{A}(\partial\Omega),$
then $W_{0}^{s,p;\,\alpha,\beta}(\Omega)=W^{s,p;\,\alpha,\beta}(\Omega)$.\\
(III) If $\Omega$ is $\kappa$-plump and $sp+\alpha+\beta>\ocodim_{A}(\partial\Omega)$,
then $W_{0}^{s,p;\,\alpha,\beta}(\Omega)\neq W^{s,p;\,\alpha,\beta}(\Omega)$.\\
\end{thm}

\begin{thm}\label{tw2}
Let $\Omega\subset \R^d$ be a nonempty, bounded, uniform and open set, let $0<s<1$, $1\leq p <\infty$ and $0\leq \alpha,\beta<\ucodim_{A}(\partial\Omega)$. If $sp+\alpha+\beta>\ocodim_{A}(\partial\Omega)$, then
$$
W_{0}^{s,p;\,\alpha,\beta}(\Omega)=\left\{f\in W^{s,p;\,\alpha,\beta}(\Omega):\int_{\Omega}\frac{|f(x)|^{p}}{d_{\Omega}(x)^{sp+\alpha+\beta}}\,dx<\infty\right\}.
$$
\end{thm}

Theorem \ref{tw2} reveals the property known partially also for classical (unweighted) Sobolev spaces $W^{1,p}(\Omega)$, see \cite[Example 9.12]{MR802206} or \cite{MR1470421}.


 \begin{rem}
 In the proof of the case II in the Theorem \ref{tw1} we use a reflexivity property of the space $W^{s,p;\,\al\,\be}(\Omega)$ (see Proposition \ref{refl}). This explains why $p=1$ is excluded from the assumptions. It is not clear if the density property holds in this case and we leave it as an open problem.
 \end{rem}  
    
To prove the case (III) of the Theorem \ref{tw1}, we use a (weak) weighted fractional Hardy inequality, which can can be easily derived from the (weak) fractional  $(s,p,a)$-Hardy inequality, given in \cite[Corollary 3]{MR3237044} and also in \cite[Theorem 5]{DK} in the case (T') of the result below. It suffices to take the function $\phi(x)=x^{sp+\alpha+\beta}$ and notice that $\dist(y,\partial\Omega)\lesssim\dist(x,\partial\Omega)$ on the ball $B(x,R\dist(x,\partial\Omega)).$ We present this version below.
\begin{thm}(\cite[Corollary 3]{MR3237044}, \cite[Theorem 5]{DK})\label{Hardy}
Let $0<p<\infty$, $0<s<1$ and $\alpha,\beta\geq 0$. Suppose that $\Omega\neq\emptyset$ is an open, $\kappa$-plump set so that either condition (T), or condition (T'), or condition (F) holds
\begin{itemize}
\item[(T)]  $sp+\alpha+\beta <\ucodim_{A}(\partial\Omega)$, $\Omega$ is unbounded and $\xi=0$,
\item[(T')]  $sp+\alpha+\beta <\ucodim_{A}(\partial\Omega)$, $\Omega$ is bounded and $\xi=1$,
\item[(F)]  $sp+\alpha+\beta >\ocodim_{A}(\partial\Omega)$, $\Omega$ is bounded or $\partial \Omega$ is unbounded and  $\xi=0$.
\end{itemize}
Then there exist constants $c$ and $R$ such that the following inequality
\begin{equation}\label{eq:fhi}
 \int_\Omega \frac{|u(x)|^p}{d_{\Omega}(x)^{sp+\alpha+\beta}}\,dx \leq
c \int_\Omega\!\int_{\Omega\cap B(x,R\,d_{\Omega}(x))}
     \frac{|u(x)-u(y)|^p}{|x-y|^{d+sp}}\frac{\,dy}{d_{\Omega}(y)^{\beta}}\frac{\,dx}{d_{\Omega}(x)^{\alpha}} + c\xi \|u\|_{L^p(\Omega)}^p,
\end{equation}
holds for all measurable functions $u$ for which the left-hand side is finite.
\end{thm}
As an easy corollary in the case (T'), deriving directly from Theorem \ref{tw1} and Theorem \ref{Hardy}, we obtain the embedding $W^{s,p;\,\alpha,\beta}(\Omega)\subset L^{p}(\Omega,\dist(\cdot,\partial\Omega)^{-sp-\alpha-\beta})$.
\begin{thm}\label{tw.emb}
Let $1\leq p<\infty$ and $0<s<1$. Suppose that $\Omega\neq\emptyset$ is an open, uniform, bounded set such that $0\leq\alpha,\beta<\ucodim_{A}(\partial\Omega)$ and $sp+\alpha+\beta <\ucodim_{A}(\partial\Omega)$. Then there exists a constant $c$ such that
$$
\int_{\Omega}\frac{|f(x)|^{p}}{d_{\Omega}(x)^{sp+\alpha+\beta}}\,dx\leq c\|f\|^{p}_{W^{s,p;\,\alpha,\beta}(\Omega)}<\infty,
$$
for all $f\in W^{s,p;\,\alpha,\beta}(\Omega)$.
\end{thm}
Theorem \ref{tw.emb} is a generalization of the unweighted case from \cite[Theorem 4]{DK}, provided that $\Omega$ is a uniform domain.


\textbf{Notation.} Having two nonnegative functions $A$ and $B$ we use a symbol ,$\lesssim$' if there exists a constant $c>0$ such that $A\leq cB$. The constant $c$ usually depends on some parameters, like $\alpha,\,\beta,\,d,\,s,\,p,\,\Omega$, but not on the arguments of the functions $A,B$ and the set of these parameters arises from context. Moreover, we write $A\approx B$ when $A\lesssim B$ and $B\lesssim A$.

\textbf{Acknowledgements.} The author would like to thank Artur Rutkowski for careful reading of the manuscript and useful remarks, in particular helpful discussions on the proof of Theorem \ref{tw.comp} and Bartłomiej Dyda for careful reading of the manuscript and valuable comments. 
    \section{Definitions}\label{section2}
    
    We will use the same definitions as in \cite[Section 2]{DK}; for Reader's
convenience we repeat them below.
    \subsection{Assouad and Minkowski dimensions}
    
    Recall that we denote  the distance from $x\in \R^d$ to a~set $E\subset \R^d$ by
$\dist(x,E)=\displaystyle\inf_{y\in E}|x-y|$.

\begin{defn}\label{defn3}
  Let $r>0$. For open sets $\Omega\subset \R^d$
  we define the \emph{inner tubular neighbourhood}  of $\Omega$ as
\[
\Omega_{r}=\left\{x\in\Omega:d_{\Omega}(x)\leq r\right\},
\]
and for arbitrary sets $E\subset \R^d$ we define the \emph{tubular neighbourhood} of $E$ as
$$
\widetilde{E}_{r}=\left\{x\in\mathbb{R}^{d}: \dist(x, E) \leq r\right\}.
$$

\end{defn}
\begin{defn}
\cite[Section 3]{MR3205534} Let $E\subset\mathbb{R}^{d}$. The \emph{lower Assouad codimension} $\underline{\text{co\,dim}}_{A}(E)$ is defined as the supremum of all $q\geq 0$, for which there exists a constant $C=C(q)\geq 1$ such that for all $x\in E$ and $0<r<R<\diam E$ it holds
$$
\left|\widetilde{E}_{r}\cap B(x,R)\right|\leq C\left|B(x,R)\right|\left(\frac{r}{R}\right)^{q}.
$$
Conversely, the \emph{upper Assouad codimension} $\overline{\text{co\,dim}}_{A}(E)$ is defined as the infimum of all $s\geq 0$, for which there exists a constant $c=c(s)>0$ such that for all $x\in E$ and $0<r<R<\diam E$ it holds
$$
\left|\widetilde{E}_{r}\cap B(x,R)\right|\geq c\left|B(x,R)\right|\left(\frac{r}{R}\right)^{s}.
$$
\end{defn}

We remark that having strict inequality $R<\diam E$ above
makes the definitions applicable also for unbounded sets $E$;
for bounded sets $E$ we could have $R\leq \diam E$.

In Euclidean space $\mathbb{R}^{d}$ (more general - in Ahlfors $d$-regular measure metric spaces) it holds $$\underline{\text{dim}}_{A}(E)+\overline{\text{co\,dim}}_{A}(E)=\overline{\text{dim}}_{A}(E)+\underline{\text{co\,dim}}_{A}(E)=d,$$ where $\underline{\text{dim}}_{A}(E) $ and $\overline{\text{dim}}_{A}(E)$ denote respectively the well known lower and upper Assouad dimension -- see for example \cite[Section 2]{MR3205534}. Moreover, if $\underline{\text{co\,dim}}_{A}(E)=\overline{\text{co\,dim}}_{A}(E),$ we simply denote both of these values by $\text{co\,dim}_{A}(E)$.

   We recall a notion of $\sigma$-homogenity, coming from \cite[Theorem A.12]{MR1608518}.
\begin{defn}
Let $E\subset\R^d$ and let $V(E,x,\lambda,r)=\{y\in\R^d:\dist(y,E)\leq r, |x-y|\leq \lambda r\}$. We say that E is $\sigma$-\emph{homogeneous}, if there exists a~constant $L$ such that
$$
|V(E,x,\lambda,r)|\leq Lr^{d}\lambda^{\sigma}
$$
for all $x\in E$, $\lambda \geq 1$ and $r>0$.
\end{defn}
If $0<r<R<\diam(E)$, then taking $\lambda=R/r$ in the definition gives
\[
  \left|\widetilde{E}_{r} \cap B(x,R)\right| = \left|V\Big(E,x,\frac{R}{r},r\Big)\right| \leq C\left|B(x,R)\right|\left(\frac{r}{R}\right)^{d-\sigma},
\]
where $C=C(d,E)$ is a constant. This means that if $\ucodim_{A}(E)=s$, then  $(d-s)$-homogeneous sets are precisely these sets $E$, for which the supremum in the definition of the lower Assouad codimension is attained. For the definition of the concept of homogenity 
from a different point of view the Reader may also see \cite[Definition 3.2]{MR1608518}.

    \begin{defn}\label{mink}
    The \emph{upper Minkowski dimension} of a set $E\subset\R^d$ is defined as
$$
\overline{\text{dim}}_{M}(E)=\inf\{s\geq0: \limsup_{r\rightarrow0}\left|\widetilde{E}_{r}\right|r^{d-s}=0\},
$$
see for example \cite[Section 2]{MR4144553}.
    \end{defn}
    It is not hard to see that $\ucodim_{A}(E)\leq d-\overline{\dim}_{M}(E)$ and the equality holds if $E$ is $(d-\ucodim_{A}(E))$ - homogeneous. Moreover (considering again open, bounded sets $\Omega$), the \emph{distance zeta function}
    $$
    \zeta_{\Omega}(q):=\int_{\Omega}\frac{dx}{d_{\Omega}(x)^{q}}
    $$
    is finite if $q<d-\overline{\dim}_{M}(\partial\Omega)$ and infinite if $q>d-\overline{\dim}_{M}(\partial\Omega)$ (see \cite[Lemma 3.3 and Lemma 3.5]{MR4144553}).

	We recall a geometric notion from \cite{10.2748/tmj/1178228081}, appearing among other assumptions in Theorem \ref{Hardy}.
	\begin{defn}
		A set $E\subset \R^d$ is {\em $\kappa$-plump}
		with $\kappa\in (0,1)$ if, for each $0<r< \diam(E)$ and each $x\in \overline{E}$, there
		is $z\in \overline{B}(x,r)$ such that
		$B(z,\kappa r)\subset E$.\end{defn}

    \subsection{Whitney decomposition and operator $P^{\eta}$}
    Let $\Omega$ be an open, nonempty, proper subset of $\R^{d}$. Let $Q$ be any closed cube in $\R^d$. We denote by $l(Q)$ the length of the side of $Q$ and by $x_{Q}$ the center of $Q$. Following \cite{MR3667439}, there exists a family of dyadic cubes $\mathcal{W}=\{Q_{n}\}_{n\in\mathbb{N}}$, called the \emph{Whitney decomposition}, satisfying for all $Q,S\in\mathcal{W}$ the conditions:
    \begin{itemize}
        \item $\Omega=\bigcup_{n}Q_{n};$
        \item if $Q\neq S$, then $\Int Q\cap \Int S=\emptyset;$
        \item there exists a constant $C=C(\mathcal{W})$ such that $C\diam Q\leq \dist(Q,\partial\Omega)\leq 4C\diam Q;$
        \item if $Q\cap S\neq\emptyset$, then $l(Q)\leq 2l(S)$;
        \item if $Q\subset 5S$ then $l(S)\leq 2l(Q)$.
    \end{itemize}
    
     The dilation of the cube $Q$, $cQ$ for $c>0$, is always taken with respect to its center, that is $cQ$ is a cube with the same center as $Q$, but the length of the side $cl(Q)$.
    
    Inspired by \cite{MR3667439} we define a \emph{shadow} of a cube $Q\in\mathcal{W}$ as 
    $$
    \Sh_{\theta}(Q)=\{S\in\mathcal{W}:S\subset B\left(x_{Q},\theta l(Q)\right)\}.
    $$
    The ,,realization" of $\Sh_{\theta}$ is $\SH_{\theta}(Q)=\bigcup\Sh_{\theta}(Q)$. When $\theta$ is fixed, we abbreviate the notation as $\Sh_{\theta}(Q)=:\Sh(Q)$ and $\SH_{\theta}(Q)=:\SH(Q)$. 
 
    For all $Q,S\in\mathcal{W}$ we define their \emph{long distance} $D$ as
    $$
    D(Q,S)=l(Q)+\dist(Q,S)+l(S).
    $$
    
    We say that a sequence of cubes $(Q,R_{1},R_{2},\dots,R_{n},S)$ is a \emph{chain}, if all two adjacent cubes have nonempty intersection. We denote $(Q,R_{1},R_{2},\dots,R_{n},S)=[Q,S]$ and $[Q,S)=[Q,S]\setminus{S}$.
    
    The Whitney decomposition is \emph{admissible}, if there exists $a>0$ such that for all $Q,S\in\mathcal{W}$ there exists a chain $[Q,S]=(Q_{1},Q_{2},\dots,Q_{n})$ satisfying
    \begin{itemize}
        \item $\displaystyle\sum_{i=1}^{n}l(Q_{i})\leq \frac{1}{a}D(Q,S)$;
        \item there exists $1\leq i_{0}\leq n$ such that $l(Q_{i})\geq a D(Q,Q_{i})$ for all $1\leq i\leq i_{0}$ and $l(Q_{i})\geq a D(Q_{i},S)$ for all $i_{0}\leq i\leq n$. We denote $Q_{i_{0}}=:Q_{S}$. This is the so-called \emph{central} cube in the chain $[Q,S]$.
    \end{itemize}
    
    As stated in \cite{MR3667439}, for a $\gamma$-admissible Whitney decomposition we can always take sufficiently large $\rho=\rho_{\gamma}>1$ such that for every $\gamma$-admissible chain of cubes $[Q,S]$ we have $Q\in\Sh_{\rho_{\gamma}}(P)$ for $P\in[Q,Q_S]$ and $5Q\subset\SH_{\rho_{\gamma}}(Q)$ for every Whitney cube $Q\in\mathcal{W}$.

    Next, we recall the definition and basic properties of the operator $P^{\eta}$, defined in \cite{MR4190640}. From now on we fix a Whitney decomposition $\mathcal{W}$ such that $C(\mathcal{W})=1$ (see \cite{MR0290095}) and $0<\varepsilon<\sqrt{5/4}-1<\frac{1}{4}$. If $Q$ is a cube, we denote by $Q^{*}$ the cube $Q$ ,,blown up'' $(1+\varepsilon)$ times, that is the cube with the same center $x_{Q^{*}}=x_{Q}$, but the length of the side $l(Q^{*})=(1+\varepsilon)l(Q)$. The cube $Q_{n}^{**}$ is defined in a similar way, that is $Q_{n}^{**}=\left(Q_{n}^{*}\right)^{*}$. Notice that our choice of $\varepsilon$ guarantees that $(1+\varepsilon)^{2}<\frac{5}{4}$ and in consequence $Q_{n}^{**}\subset \frac{5}{4}Q_{n}$. Moreover, each point $x\in\Omega$ belongs to at most $12^d$ cubes $Q_n^{**}$.
    
    Let $\{\psi_{n}\}_{n\in\mathbb{N}}$ be a partition of unity adjusted to the Whitney decomposition $\mathcal{W}=\{Q_{n}\}_{n\in\mathbb{N}}$ of $\Omega$, that is a family of functions satisfying $0\leq\psi_{n}\leq 1$, $\psi_{n}=1$ on $Q_{n}$, $\supp\psi_{n}\subset Q_{n}^{*}$, $\psi_{n}\in C_{c}^{\infty}(\Omega)$, $\sum_{n}\psi_{n}=1$ and $|\psi_{n}(x)-\psi_{n}(y)|\leq C|x-y|/l(Q_{n})$ for some positive constant $C$ independent of $Q_{n}$. Let us also fix a nonnegative function $h\colon\R^d\to\R$ with the following properties: $\supp h=B(0,1)$, $\int_{R^d}h(x)\,dx=1$, $h\in C^{\infty}(\R^d)$. For $\delta>0$ we define its dilation as $h_{\delta}(x)=\delta^{-d}h(x/\delta)$. Moreover, let $\eta\colon\mathcal{W}\to(0,\infty)$ be any function satisfying $\eta(Q)<\frac{\varepsilon}{2}l(Q)$ for all $Q\in\mathcal{W}$ (a typical example is $\eta(Q)=\delta\,l(Q)$ for any $\delta<\varepsilon/2$).   For $f\in L^{1}_{loc}(\Omega)$, extended by $0$ on $\R^d\setminus\Omega$, we define the operator $P^{\eta}$  as 
    
    $$
    P^{\eta}f=\sum_{n=1}^{\infty}(f\psi_{n})*h_{\eta(Q_{n})}.
    $$
    
    Here $f*g(x)=\int_{\R^d}f(y)g(x-y)\,dy$ is the standard convolution operation. It was proved in \cite{MR4190640} that $P^{\eta}$ is well defined, $P^{\eta}f\in C^{\infty}(\Omega)$ and $P^{\eta}$ maps the space of all compactly supported, locally integrable functions into $C_{c}^{\infty}(\Omega)$ (see \cite{MR4190640}, Propositions 1 and 2).
  
     \subsection{Uniform domains} 
     There are two equivalent ways to define the notion of uniform domain. The first one comes from \cite{10.2748/tmj/1178228081}, and the second one uses the Whitney decomposition and chains of cubes and can be found for example in \cite{MR3667439}. We present both definitions here.
     \begin{defn}
     A domain (i.e. connected, open set) $\Omega\subset\mathbb{R}^{d}$ is \emph{uniform}, if there exists a constant $C\geq 1$ such that for all points $x,y\in\Omega$ there is a curve $\gamma\colon[0,l]\to\Omega$ joining them, parameterized by arc length and satisfying $l\leq C|x-y|$ and $\dist(z,\partial\Omega)\geq \frac{1}{C}\min\{|z-x|,|z-y|\}$ for all $z\in\gamma$. Equivalently, a domain $\Omega\subset\mathbb{R}^{d}$ is uniform, if there exists an admissible Whitney decomposition of $\Omega$.
     \end{defn}
     
     Uniform domains and various reformulations of the definitions above appear also in \cite{MR581801}, \cite{MR595191} and \cite{MR565886}. To give a concrete, nontrivial example, we remark here that the Koch snowflake is known to be uniform, despite the highly irregular behaviour of its boundary. It is also $\sigma$-homogeneous with $\sigma=\log_{3}4$, according to \cite[Theorem 1.1]{MR2269586}.
      
     \subsection{Muckenhoupt class $A_{1}$ and Hardy-Littlewood maximal operator}\label{muckenhoupt}
     \begin{defn}
    For $f\in L^{1}_{loc}(\R^d)$ the (non-centered) maximal Hardy-Littlewood operator is defined as 
    $$
    Mf(x)=\sup_{Q\ni x}\frac{1}{|Q|}\int_{Q}f(y)\,dy,
    $$
    where supremum is taken over all cubes $Q$ containing $x$. Equivalently, $M$ can be defined using balls containing $x$ instead of cubes (up to a multiplicative constant). It is well known that this operator is bounded on $L^{p}(\R^d)$, whenever $1<p\leq\infty$. 
    \end{defn}
    \begin{defn}
    We say that a positive weight $w$ belongs to the \emph{Muckenhoupt class} $A_{1}$, if there exists a constant $A>0$ such that for all cubes $Q\subset\R^{d}$ it holds
    \begin{equation}\label{muck}
    \frac{1}{|Q|}\int_{Q}w(x)\,dx\leq A\displaystyle\inf_{y\in Q}w(y).
    \end{equation}
    \end{defn}
     Notice that by (\ref{muck}) we can easily see that if $w\in A_1$, then the maximal Hardy-Littlewood operator acting on the function $w$ satisfies
     \begin{equation}\label{Mw}
    Mw(x)=\sup_{Q\ni x}\frac{1}{|Q|}\int_{Q}w(y)\,dy\leq A w(x),     
     \end{equation}
     where $A$ depends on $w$. This property will be important for us later in the proof of Theorem \ref{tw.comp}. Moreover, it was proved in \cite[Theorem 1.1 (B)]{MR3900847} that the weight $d_{\Omega}^{-\alpha}$ belongs to the Muckenhoupt class $A_{1}$ if and only if $0\leq\alpha<\ucodim_{A}(\partial\Omega)$. Hence, by \eqref{Mw}, $Md_{\Omega}^{-\alpha}$ satisfies
    \begin{equation}\label{muck2}
    Md_{\Omega}^{-\alpha}(x)\leq A\,d_{\Omega}(x)^{-\alpha},   
    \end{equation}
where the constant $A$ depends on $\Omega$ and $\alpha\in[0,\ucodim_A(\partial\Omega))$.

    \section{Lemmas}\label{section3}
     We start with showing that under some assumptions $C_c^{\infty}(\Omega)$ is a subset of $W^{s,p;\,\al,\be}(\Omega)$ and in consequence the latter is not trivial. This is an analogue of \cite[Lemma 2.1]{MR3420496}, where the same fact was established for $\Omega=\R^d\setminus\{0\}$. Although we consider bounded domains, it agrees with the cited result in some aspects, as we have $\text{co\,dim}_{A}(\{0\})=d$. Noteworthy, if one of the exponents $\al,\be$ is  nonpositive, then the corresponding weight is bounded and this case is trivial.

   \begin{lem}\label{inlcusion}
    Let $\Omega\subset\R^d$ be a bounded, uniform domain. Suppose that $0<s<1$, $1\leq p<\infty$, $0\leq \al,\be<\ucodim_{A}(\partial\Omega)$ and $\al+\be<d-\overline{\emph{\text{dim}}}_M(\partial\Omega)+p(1-s)$. Then $C_c^{\infty}(\Omega)\subset W^{s,p;\,\al,\be}(\Omega)$.
   \end{lem}
    \begin{proof}
     Let $\varphi\in C_c^{\infty}(\Omega)$. Then $\varphi$ is Lipschitz and locally integrable, so, by Theorem \ref{tw.comp} with $\theta=\frac{1}{2}$ we have
     \begin{align*}
      [\varphi]^{p}_{W^{s,p;\,\al,\be}(\Omega)}&\lesssim\int_\Omega\int_{B\left(x,\frac{1}{2} d_\Omega(x)\right)}\frac{|\varphi(x)-\varphi(y)|^{p}}{|x-y|^{d+sp}}d_\Omega(x)^{-\al}d_\Omega(y)^{-\be}\,dy\,dx\\
      &\lesssim\int_{\Omega}d_\Omega(x)^{-\al-\be}\,dx\int_{B\left(x,\frac{1}{2} d_\Omega(x)\right)}\frac{|\varphi(x)-\varphi(y)|^{p}}{|x-y|^{d+sp}}\,dy\\
      &\lesssim\int_{\Omega}d_\Omega(x)^{-\al-\be}\,dx\int_{B\left(x,\frac{1}{2} d_\Omega(x)\right)}\frac{dy}{|x-y|^{d+sp-p}}\\
      &\lesssim\int_\Omega d_\Omega(x)^{-\al-\be+p(1-s)}\,dx<\infty,
     \end{align*}
     where the last inequality follows from the properties of the distance zeta function.
    \end{proof}
    \begin{rem}
    We make an easy observation that for any Borel subset $A\subset\Omega$ and $\al,\be\geq 0$ it holds
     \begin{equation}\label{2}
    \int_{A}\int_{A}\frac{|f(x)-f(y)|^{p}}{|x-y|^{d+sp}}d_\Omega(x)^{-\al}d_\Omega(y)^{-\beta}\,dy\,dx\leq 2\int_A\int_A \frac{|f(x)-f(y)|^{p}}{|x-y|^{d+sp}}d_\Omega(x)^{-\al-\be}\,dy\,dx.     
     \end{equation}
    Indeed, to prove \eqref{2} it suffices to split the inner integral into integrals over $A\cap\{d_\Omega(x)\geq d_\Omega(y)\}$ and $A\cap\{d_\Omega(x)< d_\Omega(y)\}$ and use the symmetry between variables $x$ and $y$. According to above, if we abandon the assumption about the uniformity of $\Omega$ in Lemma \ref{inlcusion}, then, using \eqref{2}, if $\Omega$ is bounded, we can analogously show that $C_c^{\infty}(\Omega)\subset W^{s,p;\,\al,\be}(\Omega)$ for $\al+\be<d-\overline{\dim}_M(\partial\Omega)$. Interestingly, this is a different range of parameters than in the Lemma \ref{inlcusion}.
    
    Moreover, if $C_{c}^{\infty}(\Omega)\subset W^{s,p;\,\al,\be}(\Omega)$ and $\Omega$ is bounded, then we cannot have $\al,\be>d-\overline{\dim}_M(\partial\Omega)$. Indeed, if $\varphi\in C_{c}^{\infty}(\Omega)$, then simple calculation shows that
    \begin{align*}
    [\varphi]^{p}_{W^{s,p;\,\al,\be}(\Omega)}&\geq \diam(\Omega)^{-d-sp}\int_\Omega\int_\Omega |\varphi(x)-\varphi(y)|^{p}d_\Omega(y)^{-\be}d_\Omega(x)^{-\al}\,dy\,dx\\
    &\geq\diam(\Omega)^{-d-sp}\int_{\supp\varphi}\int_{\Omega\setminus\supp\varphi} |\varphi(x)|^{p}d_\Omega(y)^{-\be}d_\Omega(x)^{-\al}\,dy\,dx.
    \end{align*}
    The inner integral $\int_{\Omega\setminus\supp\varphi}d_\Omega(y)^{-\beta}\,dy$ is infinite if $\beta>d-\overline{\dim}_M(\partial\Omega)$. The case when $\al>d-\overline{\dim}_M(\partial\Omega$ can be obtained similarly.
    \end{rem}
    \begin{defn}\label{loc}
    A weight $w\colon\Omega\to\R^d$ is \emph{locally comparable to a constant} if for every compact subset $K\subset\Omega$ there exists $C_{K}>0$ such that $\frac{1}{C_{K}}\leq w(x)\leq C_{K}$ for almost all $x\in K$.
    \end{defn}
    The following Theorem is a generalization of \cite[Theorem 12]{MR4190640}, where the same fact was proved for $w=v$.
    \begin{thm}\label{gen}
    Let $\Omega\subset\R^d$ be an nonempty open set, $0<s<1,\,p\in[1,\infty)$. Denote 
    $$
    \widetilde{W}^{s,p;\,w,v}(\Omega)=\left\{f\colon\Omega\to\R^d \text{ \emph{measurable}}:\int_{\Omega}\int_{\Omega}\frac{|f(x)-f(y)|^{p}}{|x-y|^{d+sp}}\,w(x)\,v(y)\,dx\,dy<\infty\right\}.
    $$
    We understand $\widetilde{W}^{s,p;\,w,v}(\Omega)$ as a semi-normed space. If $w$ and $v$ are locally bounded and satisfy the integral condition
    \begin{equation}\label{wv}
    \int_{\Omega}\frac{w(x)}{(1+|x|)^{d+sp}}\,dx<\infty,\,\int_{\Omega}\frac{v(x)}{(1+|x|)^{d+sp}}\,dx<\infty,
    \end{equation}
    then $C^{\infty}(\Omega)\cap \widetilde{W}^{s,p;\,w,v}(\Omega)$ is dense in $\widetilde{W}^{s,p;\,w,v}(\Omega)$. Moreover, we have 
     $$
   W^{s,p;\,w,v}_{0}(\Omega)=W^{s,p;\,w,v}_{c}(\Omega).
    $$
    \end{thm}
    \begin{proof}
The proof follows the proof of \cite[Theorem 12]{MR4190640}. First, we fix a Whitney decomposition $\mathcal{W}=\{Q_n\}_{n\in\mathbb{N}}$ of $\Omega$ with a constant $C(\mathcal{W})=1$. We extend $w$ and $v$ by $0$ outside $\Omega$. If $w$ or $v$ take the value zero on $\Omega$, then we can artificially augment them by adding a positive, locally comparable to a constant weights $w'$, $v'$, which in addition satisfy \eqref{wv}. New weights $w+w'$ and $v+v'$ are also locally comparable to a constant, positive and satisfy \eqref{wv}. In this case $w$ and $v$ should be replaced by $w+w'$ and $v+v'$ in all the computations below.

Denote by $\tau_{y}$ the translation operator, that is $\tau_{y}f(x)=f(x-y),\,x,y\in \R^d$, and let $M=12^{d(p-1)}$. Moreover, let $f\in \widetilde{W}^{s,p;\,w,v}(\Omega)$ and
$$
g_{n}(x,y)=\frac{f(x)\psi_{n}(x)-f(y)\psi_{n}(y)}{|x-y|^{\frac{d}{p}+s}}\ind_{\Omega\times\Omega}(x,y).
$$
We have
$$
[P^{\eta_{k}}f-f]^{p}_{W^{s,p;\, w,v}(\Omega)}\leq M\sum_{n=1}^{\infty}\int_{\R^d}\|\tau_{\eta_{k}(Q_{n})u}g_{n}-g_{n}\|^{p}_{L^{p}(\R^{2d},w\times v)}\,h(u)\,du
$$
and, for $t<\eta_{k}(Q_{n})$,
\begin{align*}
 &\|\tau_{t}g_{n}-g_{n}\|^{p}_{L^{p}(\R^{2d},w\times v)}\\
 &\leq \int_{Q_{n}^{*}}\int_{Q_{n}^{**}}\frac{|f(x-t)\psi_{n}(x-t)-f(y-t)\psi_{n}(y-t)-f(x)\psi_{n}(x)+f(y)\psi_{n}(y)|^{p}}{|x-y|^{d+sp}}\,w(x)\,v(y)\,dx\,dy\\
 &+\int_{Q_{n}^{*}}\int_{\Omega\setminus Q_{n}^{**}}\frac{|f(x)\psi_{n}(x)-f(x-t)\psi_{n}(x-t)|^{p}}{|x-y|^{d+sp}}\,w(x)\,v(y)\,dx\,dy\\
 &+\int_{Q_{n}^{*}}\int_{\Omega\setminus Q_{n}^{**}}\frac{|f(x)\psi_{n}(x)-f(x-t)\psi_{n}(x-t)|^{p}}{|x-y|^{d+sp}}\,w(y)\,v(x)\,dx\,dy\\
 &=:I_{1}+I_{2}+I_{3}.
\end{align*}

The estimates of the integrals $I_{1}\,,I_{2}$ and $I_{3}$ and completely analogous to these from \cite[Proof of Theorem 12]{DK}. Notice that the properly modified version of \cite[Proposition 9]{MR4190640} also holds. The equality between $W_{0}^{s,p;\, w,v}(\Omega)$ and $W_{c}^{s,p;\, w,v}(\Omega)$ is a consequence of \cite[Proposition 2]{MR4190640} and the fact that the approximating functions are of the form $P^{\eta_{k}}f$.
    \end{proof}
    \begin{rem}
    Suppose that $\Omega$ is bounded. Then we trivially have 
    $$1\leq (1+|x|)^{d+sp}\leq M:= \displaystyle\sup_{x\in\Omega}(1+|x|)^{d+sp}<\infty,$$
    hence, the condition \eqref{wv} is equivalent to $w,v\in L^{1}(\Omega)$. Moreover, if $w(x)=d_{\Omega}(x)^{-\alpha},$ $v(x)=d_{\Omega}(x)^{-\beta}$, then \eqref{wv} is satisfied when $0\leq\alpha,\beta<d-\overline{\dim}_{M}(\partial\Omega)$ (we refer again to \cite{MR4144553}). Of course, the function $d_{\Omega}(x)^{-a}$ is locally comparable to a constant on $\Omega$ for every $a\in\R$.
    \end{rem}
 
    \begin{lem}\label{ind}
    Let $\Omega\subset\R^d$ be a nonempty, open set such that $|\Omega|<\infty$. Then we have
    $$
    W^{s,p;\,w,v}_{0}(\Omega)=W^{s,p;\,w,v}(\Omega)\Longleftrightarrow \ind_{
    \Omega}\in W^{s,p;\,w,v}_{0}(\Omega).
    $$
    \end{lem}
    \begin{proof}
    Using the result of Theorem \ref{gen} about the equality between $W^{s,p;\,w,v}_{0}(\Omega)$ and $W^{s,p;\,w,v}_{c}(\Omega)$, the proof is a copy of \cite[Lemma 13]{DK}.
    \end{proof}
    \begin{lem}\label{ineq}
     Let $\Omega$ be an open, uniform, bounded domain and let
\[
  v_{n}(x)= \max\left\{\min\left\{2 - nd_{\Omega}(x),1\right\},0\right\}=\left\{ \begin{array}{ll}
1 & \textrm{when $d_{\Omega}(x)\leq 1/n$,}\\
2 - nd_{\Omega}(x) & \textrm{when $ 1/n < d_{\Omega}(x)\leq2/n$,}\\
0 & \textrm{when $d_{\Omega}(x)>2/n$.}
  \end{array} \right.
  \]
  There exists a constant $C=C(d,s,p,\alpha,
  \beta,\Omega)>0$ such that the following inequality holds
  for all functions $f\in W^{s,p;\,\alpha,\beta}(\Omega)$ and $0\leq\alpha,\beta<\ucodim_{A}(\partial\Omega)$,
  \begin{equation}\label{eq:fvn}
    [fv_{n}]^{p}_{W^{s,p;\,\alpha,
    \beta}(\Omega)}
    \leq Cn^{sp}\int_{\Omega_{\frac{3}{n}}}\frac{|f(x)|^{p}}{d_{\Omega}(x)^{\alpha+\beta}}\,dx+C\int_{\Omega_{\frac{3}{n}}}\int_{\Omega_{\frac{3}{n}}}\frac{|f(x)-f(y)|^{p}}{|x-y|^{d+sp}}d_{\Omega}(x)^{-\alpha}d_{\Omega}(y)^{-\beta}\,dy\,dx.
\end{equation}
Moreover, without assuming the uniformity of $\Omega$, the following weaker inequality is satisfied for all $\al,\be\geq 0$, $\al+\be<d-\overline{\dim}_M(\partial\Omega)$ and $f\in L^{\infty}(\Omega)$, 
    \begin{equation}\label{eq:fvn2}
    [fv_{n}]^{p}_{W^{s,p;\,\alpha,
    \beta}(\Omega)}
    \leq C\|f\|^{p}_\infty n^{sp}\int_{\Omega_{\frac{3}{n}}}\frac{dx}{d_{\Omega}(x)^{\alpha+\beta}}+C\int_{\Omega_{\frac{3}{n}}}\int_{\Omega_{\frac{3}{n}}}\frac{|f(x)-f(y)|^{p}}{|x-y|^{d+sp}}d_{\Omega}(x)^{-\alpha}d_{\Omega}(y)^{-\beta}\,dy\,dx.
    \end{equation}
    \end{lem}

    \begin{proof}
 The following proof is a modification of \cite[Lemma 10]{DK}. By Theorem \ref{tw.comp}, taking $\theta=\frac{1}{2}$ we have
    \begin{align*}
       [fv_{n}]^{p}_{W^{s,p;\,\alpha,\beta}(\Omega)}&\lesssim\int_{\Omega}\int_{B\left(x,\frac{1}{2}d_{\Omega}(x)\right)}\frac{|f(x)v_{n}(x)-f(y)v_{n}(y)|^{p}}{|x-y|^{d+sp}}d_{\Omega}(x)^{-\alpha}d_{\Omega}(y)^{-\beta}\,dy\,dx\\  
       &=\int_{\Omega_{\frac{3}{n}}}\int_{B\left(x,\frac{1}{2}d_{\Omega}(x)\right)\cap\Omega_{\frac{3}{n}}}\frac{|f(x)v_{n}(x)-f(y)v_{n}(y)|^{p}}{|x-y|^{d+sp}}d_{\Omega}(x)^{-\alpha}d_{\Omega}(y)^{-\beta}\,dy\,dx\\
       &+\int_{\Omega_{\frac{3}{n}}}\int_{B\left(x,\frac{1}{2}d_{\Omega}(x)\right)\cap\left(\Omega\setminus\Omega_{\frac{3}{n}}\right)}\frac{|f(x)v_{n}(x)|^{p}}{|x-y|^{d+sp}}d_{\Omega}(x)^{-\alpha}d_{\Omega}(y)^{-\beta}\,dy\,dx\\
       &+\int_{\Omega\setminus\Omega_{\frac{3}{n}}}\int_{B\left(x,\frac{1}{2}d_{\Omega}(x)\right)\cap\Omega_{\frac{3}{n}}}\frac{|f(y)v_{n}(y)|^{p}}{|x-y|^{d+sp}}d_{\Omega}(x)^{-\al}d_{\Omega}(y)^{-\be}\,dy\,dx\\
       &=:J_{1}+J_{2}+J_{3}. 
    \end{align*}
   Starting with estimating the integral $J_{1}$, we obtain
    \begin{align*}
    J_{1}&\lesssim\int_{\Omega_{\frac{3}{n}}}\int_{B\left(x,\frac{1}{2}d_{\Omega}(x)\right)\cap\Omega_{\frac{3}{n}}}\frac{|v_{n}(y)|^{p}|f(x)-f(y)|^{p}}{|x-y|^{d+sp}}d_{\Omega}(x)^{-\alpha}d_{\Omega}(y)^{-\beta}\,dy\,dx\\
    &+\int_{\Omega_{\frac{3}{n}}}\int_{B\left(x,\frac{1}{2}d_{\Omega}(x)\right)\cap\Omega_{\frac{3}{n}}}\frac{|f(x)|^{p}|v_{n}(x)-v_{n}(y)|^{p}}{|x-y|^{d+sp}}d_{\Omega}(x)^{-\alpha}d_{\Omega}(y)^{-\beta}\,dy\,dx\\
    &=:K_{1}+K_{2}.
    \end{align*}
    The integral $K_{1}$ can be trivially bounded from above by the remainder of the weighted Gagliardo seminorm, that is $$K_1\leq\int_{\Omega_{\frac{3}{n}}}\int_{\Omega_{\frac{3}{n}}}\frac{|f(x)-f(y)|^{p}}{|x-y|^{d+sp}}d_{\Omega}(x)^{-\alpha}d_{\Omega}(y)^{-\beta}\,dy\,dx.$$ 
    Moreover, using the bound $|v_{n}(x)-v_{n}(y)|\leq\min\{1,|x-y|\}$ and the fact that $d_{\Omega}(x)\approx d_{\Omega}(y)$ on the ball $B\left(x,\frac{1}{2}d_{\Omega}(x)\right)$ we can estimate $K_{2}$ as follows,
    \begin{align*}
    K_{2}&\lesssim\int_{\Omega_{\frac{3}{n}}}\int_{\Omega_{\frac{3}{n}}}\frac{|f(x)|^{p}\left(\min\{1,|x-y|\}\right)^{p}}{|x-y|^{d+sp}}d_{\Omega}(x)^{-\alpha-\beta}\,dy\,dx.    
    \end{align*}
    Splitting the inner integral over $dy$ into $|x-y|>1/n$ and $|x-y|\leq1/n$ gives the first term in (\ref{eq:fvn}).
    
    Going back to the integral $J_{2}$ and remembering that $v_{n}=0$ on $\Omega_{\frac{3}{n}}\setminus\Omega_{\frac{2}{n}}$ we have
    \begin{align*}
        J_{2}&=\int_{\Omega_{\frac{2}{n}}}\int_{B\left(x,\frac{1}{2}d_{\Omega}(x)\right)\cap\left(\Omega\setminus\Omega_{\frac{3}{n}}\right)}\frac{|f(x)v_{n}(x)|^{p}}{|x-y|^{d+sp}}d_{\Omega}(x)^{-\alpha}d_{\Omega}(y)^{-\beta}\,dy\,dx\\
        &\lesssim\int_{\Omega_{\frac{2}{n}}}\int_{B\left(x,\frac{1}{2}d_{\Omega}(x)\right)\cap\left(\Omega\setminus\Omega_{\frac{3}{n}}\right)}\frac{|f(x)v_{n}(x)|^{p}}{|x-y|^{d+sp}}d_{\Omega}(x)^{-\alpha-\beta}\,dy\,dx\\
        &\leq \int_{\Omega_{\frac{2}{n}}}\int_{\Omega\setminus\Omega_{\frac{3}{n}}}\frac{|f(x)v_{n}(x)|^{p}}{|x-y|^{d+sp}}d_{\Omega}(x)^{-\alpha-\beta}\,dy\,dx\\
        &\leq\int_{\Omega_{\frac{2}{n}}}|f(x)|^{p}d_{\Omega}(x)^{-\alpha-\beta}\,dx\int_{B(x,1/n)^{c}}\frac{\,dy}{|x-y|^{d+sp}}\\
        &\lesssim n^{sp}\int_{\Omega_{\frac{2}{n}}}|f(x)|^{p}d_{\Omega}(x)^{-\alpha-\beta}\,dx.\end{align*}
        
    The integral $J_{3}$ can be estimated in the similar way as $J_2$. That ends the proof of \eqref{eq:fvn}. We note that the proof of \eqref{eq:fvn2} is analogous to the previous part. $K_1$ estimates by \eqref{2} and in the integrals $J_2$ and $J_3$ we use the fact that $d_\Omega(y)\geq d_\Omega(x)$ for $y\notin \Omega_{\frac{3}{n}}$ and $x\in\Omega_{\frac{3}{n}}$, hence, the comparability is not necessary here. The only thing that essentially changes is the estimation of $K_2$. In this case we bound $|f(x)|$ from above by its $L^{\infty}$-norm, use \eqref{2} and then proceed similarly as before to obtain the desired result. That proves \eqref{eq:fvn2}.
    \end{proof}
    \section{Proof of the comparability}\label{section4}
    \begin{proof}[Proof of Theorem \ref{tw.comp}]
    In the proof of this Theorem we use techniques coming from \cite{MR3667439}. We start with fixing sufficiently fragmented Whitney decomposition $\mathcal{W}=\mathcal{W}(\theta)$, so that for $(x,y)\in Q\times 5Q$ it holds $y\in B(x,\theta\,d_{\Omega}(x))$. Suppose first that $p>1$. Let $q=\frac{p}{p-1}$ be the Hölder conjugate exponent to $p$. Using the duality between spaces $L^{p}(\Omega\times\Omega)$ and $L^{q}(\Omega\times\Omega)$ we can write the weighted Gagliardo seminorm wherewithal dual norm, that is
    \begin{align*}
    &\left(\int_{\Omega}\int_{\Omega}\frac{|f(x)-f(y)|^{p}}{|x-y|^{d+sp}}d_{\Omega}(y)^{-\beta}d_{\Omega}(x)^{-\alpha}\,dx\right)^{\frac{1}{p}}\\
   &=\sup \int_{\Omega}\int_{\Omega}\frac{|f(x)-f(y)|}{|x-y|^{\frac{d}{p}+s}}d_{\Omega}(x)^{-\frac{\alpha}{p}}d_{\Omega}(y)^{-\frac{\beta}{p}}g(x,y)\,dy\,dx,
    \end{align*}
    where the supremum is taken over all nonnegative $g\in L^{q}(\Omega\times\Omega)$ satisfying $\|g\|_{L^{q}(\Omega\times\Omega)}\leq 1$. For now on we fix such a function $g$. Now, we split the integration range as follows,
    \begin{align*}
    &\int_{\Omega}\int_{\Omega}=\sum_{Q}\int_{Q}\int_{2Q}+\sum_{Q,S}\int_{Q}\int_{S\setminus 2Q}=:S_1+S_2.
    \end{align*}

    Thanks to our assumption about the Whitney decomposition, the first sum can be immediately estimated by the truncated seminorm with making use of the Hölder inequality,
    \begin{align*}
    S_1&\leq\left(\sum_Q\int_Q\int_{2Q}\frac{|f(x)-f(y)|^{p}}{|x-y|^{d+sp}}d_\Omega(x)^{-\al}d_\Omega(y)^{-\be}\,dy\,dx\right)^{\frac{1}{p}}\|g\|_{L^{q}(\Omega\times\Omega)}\\
    &\leq\left(\sum_Q\int_Q\int_{2Q}\frac{|f(x)-f(y)|^{p}}{|x-y|^{d+sp}}d_\Omega(x)^{-\al}d_\Omega(y)^{-\be}\,dy\,dx\right)^{\frac{1}{p}}\\
    &\leq\left(\int_\Omega\int_{B(x,\theta d_\Omega(x))}\frac{|f(x)-f(y)|^{p}}{|x-y|^{d+sp}}d_\Omega(x)^{-\al}d_\Omega(y)^{-\be}\,dy\,dx\right)^{\frac{1}{p}}.
    \end{align*}
    
    Hence, we only need to estimate the second part, $S_2$. We denote by $f_{Q}$ the average value of $f$ on the cube $Q$, that is $f_{Q}=\frac{1}{|Q|}\int_{Q}f(x)\,dx$ (the latter is finite by assumption). Using similar arguments as in \cite[Section 4]{MR3667439} we observe that for $x\in Q$ and $y\in S\setminus 2Q$ it holds $|x-y|\approx D(Q,S)$, hence, triangle inequality yields
   \begin{align*}
    S_2&\lesssim \sum_Q\sum_S\int_Q\int_S\frac{|f(x)-f(y)|}{D(Q,S)^{\frac{d}{p}+s}}d_{\Omega}(x)^{-\frac{\alpha}{p}}d_{\Omega}(y)^{-\frac{\beta}{p}}g(x,y)\,dy\,dx\\
    &\leq\sum_Q\sum_S\int_Q\int_S\frac{|f(x)-f_Q|}{D(Q,S)^{d+sp}}d_{\Omega}(x)^{-\frac{\alpha}{p}}d_{\Omega}(y)^{-\frac{\beta}{p}}g(x,y)\,dy\,dx\\
    &+\sum_Q\sum_S\int_Q\int_S\frac{|f_Q-f_{Q_S}|}{D(Q,S)^{d+sp}}d_{\Omega}(x)^{-\frac{\alpha}{p}}d_{\Omega}(y)^{-\frac{\beta}{p}}g(x,y)\,dy\,dx\\
    &+\sum_Q\sum_S\int_Q\int_S\frac{|f_{Q_S}-f_{S}|}{D(Q,S)^{d+sp}}d_{\Omega}(x)^{-\frac{\alpha}{p}}d_{\Omega}(y)^{-\frac{\beta}{p}}g(x,y)\,dy\,dx\\
    &+\sum_Q\sum_S\int_Q\int_S\frac{|f_S-f(y)|}{D(Q,S)^{d+sp}}d_{\Omega}(x)^{-\frac{\alpha}{p}}d_{\Omega}(y)^{-\frac{\beta}{p}}g(x,y)\,dy\,dx\\
    &=:\boldsymbol{(A)}+\boldsymbol{(B)}+\boldsymbol{(C)}+\boldsymbol{(D)}.
   \end{align*} 
   Let us estimate $\boldsymbol{(A)}$ first. By Hölder inequality and Fubini-Tonelli theorem we get
   \begin{align*}
    \boldsymbol{(A)}&\leq\sum_{Q}\int_{Q}|f(x)-f_{Q}|d_{\Omega}(x)^{-\frac{\alpha}{p}}\left(\sum_{S}\int_{S}g(x,y)^{q}\,dy\right)^{\frac{1}{q}}\left(\sum_{S}\int_{S}\frac{d_\Omega(y)^{-\beta}}{D(Q,S)^{d+sp}}\right)^{\frac{1}{p}}\,dx\\
    &\leq\left(\sum_Q\int_Q|f(x)-f_Q|^{p}d_\Omega(x)^{-\al}\,\sum_{S}\int_{S}\frac{d_\Omega(y)^{-\beta}}{D(Q,S)^{d+sp}}\,dx\right)^\frac{1}{p}.
   \end{align*}
   By \cite[Lemma 2.7]{MR3667439} with $r=l(Q)$ and the Muckenhoupt condition \eqref{muck2}  we have
  \begin{align*}
   \sum_S\int_S\frac{d_\Omega(y)^{-\be}}{D(Q,S)^{d+sp}}\,dy&\lesssim l(Q)^{-sp}\inf_{y\in Q}Md_\Omega^{-\be}(y)\\
   &\lesssim l(Q)^{-sp}\inf_{y\in Q} d_\Omega(y)^{-\be}\\
   &\lesssim l(Q)^{-sp} d_\Omega(y)^{-\be}
  \end{align*}
 for any $y\in Q$, where $M$ is the Hardy-Littlewood maximal function. Hence, by Jensen inequality and Whitney decomposition properties, $\boldsymbol{(A)}$ can be bounded from above as follows,
    \begin{align*}
        \boldsymbol{(A)}^p&\lesssim\sum_Q\int_Q\frac{1}{|Q|}\int_Q|f(x)-f(y)|^{p}d_\Omega(x)^{-\al}d_\Omega(y)^{-\be}l(Q)^{-sp}\,dy\,dx\\
        &\lesssim\sum_Q\int_Q\int_Q|f(x)-f(y)|^{p}d_\Omega(x)^{-\al}d_\Omega(y)^{-\be}l(Q)^{-sp-d}\,dy\,dx\\
        &\lesssim\sum_Q\int_Q\int_Q\frac{|f(x)-f(y)|^{p}}{|x-y|^{d+sp}}d_\Omega(x)^{-\al}d_\Omega(y)^{-\be}\,dy\,dx\\
        &\leq \int_\Omega\int_{B\left(x,\theta d_\Omega(x)\right)}\frac{|f(x)-f(y)|^{p}}{|x-y|^{d+sp}}d_\Omega(x)^{-\al}d_\Omega(y)^{-\be}\,dy\,dx,
    \end{align*}
  as it holds $|x-y|\lesssim l(Q)$ for $x,y\in Q$.
   
Now, we face the estimation of the component $\boldsymbol{(B)}$. We denote by $\mathcal{N}(P)$ the successor of the cube $P$ in the chain $[Q,S]$. It holds $\mathcal{N}(P)\subset 5P$ and $Q\in \Sh(P)$ for $P\in[Q,Q_S]$. Also, $D(Q,S)\approx D(P,S)$ for such $P$. Hence, analogously to \cite{MR3667439}, by triangle inequality and Jensen inequality we can estimate $\boldsymbol{(B)}$ as follows,

  \begin{align*}
    \boldsymbol{(B)}&\leq\sum_{Q,S}\int_{Q}\int_{S}\frac{d_{\Omega}(x)^{-\frac{\alpha}{p}}d_{\Omega}(y)^{-\frac{\beta}{p}}}{D(Q,S)^{\frac{d}{q}+s}}g(x,y)\sum_{P\in[Q,Q_{S})}|f_{P}-f_{\mathcal{N}(P)}|\,dy\,dx\\
    &\leq\sum_{Q,S}\int_{Q}\int_{S}\sum_{P\in[Q,Q_{S})}\frac{1}{|P|}\frac{1}{|\mathcal{N}(P)|}\int_{P}\int_{\mathcal{N}(P)}|f(\xi)-f(\zeta)|\,d\xi\,d\zeta\frac{d_{\Omega}(x)^{-\frac{\alpha}{p}}d_{\Omega}(y)^{-\frac{\beta}{p}}}{D(Q,S)^{\frac{d}{p}+s}}g(x,y)\,dy\,dx\\
    &\lesssim\sum_{P}\frac{1}{|P||5P|}\int_{P}\int_{5P}|f(\xi)-f(\zeta)|\,d\xi\,d\zeta\sum_{Q\in \Sh(P)}\sum_{S}\int_{Q}\int_{S}\frac{d_{\Omega}(x)^{-\frac{\alpha}{p}}d_{\Omega}(y)^{-\frac{\beta}{p}}}{D(P,S)^{\frac{d}{q}+s}}g(x,y)\,dy\,dx.
    \end{align*}
Define $$G(x)=\left(\int_{\Omega}g(x,y)^{q}\,dy\right)^{\frac{1}{q}},\,x\in\Omega.$$ 
    Using again Hölder inequality, Muckenhoupt condition (\ref{muck2}) and Whitney covering properties we have
\begin{align*}
    &\sum_{Q\in \Sh(P)}\sum_{S}\int_{Q}\int_{S}\frac{d_{\Omega}(x)^{-\frac{\alpha}{p}}d_{\Omega}(y)^{-\frac{\beta}{p}}}{D(P,S)^{\frac{d}{p}+s}}g(x,y)\,dy\,dx\\   
   &\leq\sum_{Q\in\Sh(P)}\int_{Q}d_{\Omega}(x)^{-\frac{\alpha}{p}}\left(\sum_{S}\int_S\frac{d_\Omega(y)^{-\beta}}{D(P,S)^{d+sp}}\right)^{\frac{1}{p}}G(x)\,dx\\
   &\lesssim l(P)^{-\frac{\beta}{p}-s}\int_{\SH(P)}G(x)d_{\Omega}(x)^{-\frac{\alpha}{p}}\,dx.\\
\end{align*}
Let us take small $\varepsilon>0$, to be established in a moment. We apply Hölder inequality with exponents $q-\varepsilon$ and $\frac{q-\varepsilon}{q-\varepsilon-1}$ to the integral above to obtain
\begin{align*}
\int_{\SH(P)}G(x)d_{\Omega}(x)^{-\frac{\alpha}{p}}\,dx&\leq \left(\int_{\SH(P)}G(x)^{q-\varepsilon}(x)\,dx\right)^{\frac{1}{q-\varepsilon}}\left(\int_{\SH(P)}d_\Omega(x)^{-\frac{\alpha(q-\varepsilon)}{p(q-\varepsilon-1)}}\,dx\right)^{\frac{q-\varepsilon-1}{q-\varepsilon}}.     
\end{align*}
Notice that $$\lim_{\varepsilon\rightarrow 0^{+}}\frac{q-\varepsilon}{p(q-\varepsilon-1)}=\frac{q}{p(q-1)}=1,$$
hence, remembering that by assumption $0\leq\al<\ucodim_A(\partial\Omega)$, for sufficiently small $\varepsilon$ we still have $0\leq\frac{\alpha(q-\varepsilon)}{p(q-\varepsilon-1)}<\ucodim_A(\partial\Omega)$ (this condition defines $\varepsilon$, as well as $q-\varepsilon>1$). According to this, by \cite[Lemma 2.7]{MR3667439} and \eqref{muck2}, we have 
\begin{align*}
&\left(\int_{\SH(P)}G(x)^{q-\varepsilon}(x)\,dx\right)^{\frac{1}{q-\varepsilon}}\left(\int_{\SH(P)}d_\Omega(x)^{-\frac{\alpha(q-\varepsilon)}{p(q-\varepsilon-1)}}\,dx\right)^{\frac{q-\varepsilon-1}{q-\varepsilon}}\\
&\lesssim\left(l(P)^d\inf_{x\in P} MG^{q-\varepsilon}(x)\right)^\frac{1}{q-\varepsilon}\left(l(P)^{d}\inf_{x\in P}Md_\Omega^{-\frac{\alpha(q-\varepsilon)}{p(q-\varepsilon-1)}}(x)\right)^{\frac{q-\varepsilon-1}{q-\varepsilon}}\\
&\lesssim\left(l(P)^d\inf_{x\in P} MG^{q-\varepsilon}(x)\right)^\frac{1}{q-\varepsilon}\left(l(P)^{d-\frac{\alpha(q-\varepsilon)}{p(q-\varepsilon-1)}}\right)^{\frac{q-\varepsilon-1}{q-\varepsilon}}\\
&\leq l(P)^{d-\frac{\al}{p}}\left(MG^{q-\varepsilon}(\zeta)\right)^{\frac{1}{q-\varepsilon}}
\end{align*}
for any $\zeta\in P$. Finally, summing up all the considerations above, by Jensen inequality, Hölder inequality and boundendess of the Hardy-Littlewood maximal function on $L^{\frac{q}{q-\varepsilon}}(\R^d)$ we get the following result,  
\begin{align*}
\boldsymbol{(B)}&\lesssim\sum_{P}\frac{l(P)^{d-\frac{\alpha}{p}-\frac{\beta}{p}-s}}{|P||5P|}\int_{P}\int_{5P}|f(\xi)-f(\zeta)|\left(MG^{q-\varepsilon}(\zeta)\right)^{\frac{1}{q-\varepsilon}}\,d\xi\,d\zeta\\
&=\sum_{P}\frac{l(P)^{-\frac{\alpha}{p}-\frac{\beta}{p}-s}}{|5P|}\int_{P}\int_{5P}|f(\xi)-f(\zeta)|\left(MG^{q-\varepsilon}(\zeta)\right)^{\frac{1}{q-\varepsilon}}\,d\xi\,d\zeta\\ 
&\leq\sum_P l(P)^{-\frac{\alpha}{p}-\frac{\beta}{p}-s}\left(\int_P\left(\frac{1}{|5P|}\int_{5P}|f(\xi)-f(\zeta)|\,d\xi\right)^{p}\,d\zeta\right)^{\frac{1}{p}} \left(\int_P \left(MG^{q-\varepsilon}(\zeta)\right)^{\frac{q}{q-\varepsilon}}\,d\zeta\right)^{\frac{1}{q}}\\
&\leq\sum_P l(P)^{-\frac{\alpha}{p}-\frac{\beta}{p}-s}\left(\int_P\frac{1}{|5P|}\int_{5P}|f(\xi)-f(\zeta)|^{p}\,d\xi\,d\zeta\right)^{\frac{1}{p}} \left(\int_P \left(MG^{q-\varepsilon}(\zeta)\right)^{\frac{q}{q-\varepsilon}}\,d\zeta\right)^{\frac{1}{q}}\\ 
&\lesssim\left(\sum_P\int_P\int_{5P}\frac{|f(\xi)-f(\zeta)|^{p}}{|\xi-\zeta|^{d+sp}}d_\Omega(\zeta)^{-\al}d_\Omega(\xi)^{-\be}\right)^{\frac{1}{p}}\left(\int_{\Omega}\left(MG^{q-\varepsilon}(\zeta)\right)^{\frac{q}{q-\varepsilon}}\,d\zeta\right)^{\frac{1}{q}}\\
&\lesssim\left(\sum_P\int_P\int_{5P}\frac{|f(\xi)-f(\zeta)|^{p}}{|\xi-\zeta|^{d+sp}}d_\Omega(\zeta)^{-\al}d_\Omega(\xi)^{-\be}\right)^{\frac{1}{p}}\left(\int_{\Omega}G^{q}(\zeta)\,d\zeta\right)^{\frac{1}{q}}\\
&\lesssim\left(\sum_P\int_P\int_{5P}\frac{|f(\xi)-f(\zeta)|^{p}}{|\xi-\zeta|^{d+sp}}d_\Omega(\zeta)^{-\al}d_\Omega(\xi)^{-\be}\,d\xi\,d\zeta\right)^{\frac{1}{p}}\\
&\lesssim\left(\int_\Omega\int_{B(x,\theta d_\Omega(x))}\frac{|f(\xi)-f(\zeta)|^{p}}{|\xi-\zeta|^{d+sp}}d_\Omega(\zeta)^{-\al}d_\Omega(\xi)^{-\be}\,d\xi\,d\zeta\right)^{\frac{1}{p}}.
\end{align*}
That ends $\boldsymbol{(B)}$. We observe that the case $\boldsymbol{(C)}$ is symmetric to $\boldsymbol{(B)}$ (as we may have $Q_S=S_Q$). We will obtain the same estimate as in $\boldsymbol{(B)}$, but with $\al$ and $\be$ changed, however it holds $d_{\Omega}(x)\approx d_\Omega(y)$ for $x,y\in 5P$, hence, we will obtain exactly the same bound. The case $\boldsymbol{(D)}$ is symmetric to $\boldsymbol{(A)}$. That ends the proof in the case $p>1$. When $p=1$, we proceed similarly and actually this case is simpler and does not require the usage of dual norms.\qedhere
 \end{proof}
 When $p=1$, we can formulate even a more general result.
 \begin{thm}\label{tw.comp2}
 Let $\Omega$ be a nonempty, bounded, uniform domain and $0<s<1$, $0<\theta\leq 1$. If the weights $w,v$ belong to the Muckenhoupt class $A_1$, then the full seminorm $[f]_{W^{s,1; w\,v}(\Omega)}$ and the truncated seminorm
 $$
 \int_\Omega\int_{B\left(x,\theta d_\Omega(x)\right)}\frac{|f(x)-f(y)|}{|x-y|^{d+s}}\left(w(x)v(y)+w(y)v(x)\right)\,dy\,dx
 $$
 are comparable for all $f\in L^{1}_{loc}(\Omega)$. The comparability constant depends on $\Omega,s,d,w,v$ and $\theta$.
 \end{thm}
 \begin{proof}
The proof is similar to the proof of Theorem \ref{tw.comp}. The additional term in the truncated seminorm above follows from the fact, that components $\boldsymbol{(B)}$ and $\boldsymbol{(C)}$ are symmetric with respect to $w$ and $v$, but we cannot use the comparability of $w(x)$ and $v(y)$ on the cube $5P$, as for the distance weights.       
 \end{proof}
 \begin{rem}
 Interestingly, the result of Theorem \ref{tw.comp} allow to deduce in some cases another comparability property, between weighted Gagliardo seminorms $[f]_{W^{s,p;\,\al\,\be}(\Omega)}$ and $[f]_{W^{s,p;\,\al+\be,0}(\Omega)}$. Suppose that $\Omega$ is a nonempty, bounded, uniform domain and the parameters $\al,\be$ satisfy $0\leq\al,\be,\al+\be<\ucodim_A(\partial\Omega)$. Take $f\in L^{1}_{loc}(\Omega)$. By \eqref{2} we have 
 $$
 [f]_{W^{s,p;\,\al\,\be}(\Omega)}\leq 2^{\frac{1}{p}}[f]_{W^{s,p;\,\al+\be,0}(\Omega)}.
 $$
 To obtain a converse inequality, we use Theorem \ref{tw.comp} with $\theta=\frac{1}{2}$ and get
 \begin{align*}
  [f]^{p}_{W^{s,p;\,\al\,\be}(\Omega)}&\geq\int_\Omega\int_{B\left(x,\frac{1}{2}d_\Omega(x)\right)}\frac{|f(x)-f(y)|^{p}}{|x-y|^{d+sp}}d_\Omega(y)^{-\be}d_\Omega(x)^{-\al}\,dy\,dx\\
  &\approx \int_\Omega\int_{B\left(x,\frac{1}{2}d_\Omega(x)\right)}\frac{|f(x)-f(y)|^{p}}{|x-y|^{d+sp}}d_\Omega(y)^{-\al-\be}\,dy\,dx\\
  &\gtrsim [f]^{p}_{W^{s,p;\,\al+\be,0}(\Omega)}.
 \end{align*}
  Overall, we indeed get that 
 $$
 [f]_{W^{s,p;\,\al\,\be}(\Omega)}\approx [f]_{W^{s,p;\,\al+\be,0}(\Omega)}.
 $$
 \end{rem}
         \section{Proofs of main results}\label{sectionmainresults}
       Before we proceed to prove our main results, we need the following Proposition.
       \begin{prop}\label{refl}
       Let $\Omega$ be a nonempty open set. Then the space $W^{s,p;\,w,\, v}(\Omega)$ is reflexive for $0<s<1$, $1<p<\infty$ and all weights $w$ and $v$.
       \end{prop}
       \begin{proof}
       The proof is a modification of the proof of the reflexivity of the classical Sobolev space $W^{1,p}(\Omega)$ from \cite[Proposition 8.1]{MR2759829}. We define the isometry $T\colon W^{s,p;\,w,\, v}(\Omega)\to L^{p}(\Omega)\times L^{p}(\Omega\times\Omega,w\times v)$ (the latter endowed with the natural product norm) by
       $$
       T(u)=\left(u,\frac{u(x)-u(y)}{|x-y|^{\frac{d}{p}+s}}\right).
       $$
       The reflexitiy of $W^{s,p;\,w,\, v}(\Omega)$ is a consequence of reflexivity of $L^{p}(\Omega)\times L^{p}(\Omega\times\Omega,w\times v)$.
       \end{proof}
       \begin{proof}[Proof of Theorem~\ref{tw1}, case I]
       By Lemma \ref{ind} and Theorem \ref{gen} to prove the density of $C_{c}^{\infty}(\Omega)$ in $W^{s,p;\,\alpha,\beta}(\Omega)$ it suffices to approximate the function $f=\ind_{\Omega}$ by functions with compact support. By \eqref{eq:fvn2} (keeping the same notation), we have
       \begin{align*}
       [fv_{n}]^{p}_{W^{s,p;\,\alpha,
    \beta}(\Omega)}
    \leq Cn^{sp}\int_{\Omega_{\frac{3}{n}}}\frac{dx}{d_{\Omega}(x)^{\alpha+\beta}}   .
       \end{align*}
       
       We have
       \begin{align*}
        n^{sp}\int_{\Omega_{\frac{3}{n}}}d_{\Omega}(x)^{-\alpha-\beta}\,dx&\lesssim \int_{\Omega_{\frac{3}{n}}}d_\Omega(x)^{-\al-\be-sp}\,dx\longrightarrow 0,
       \end{align*}
       when $n\longrightarrow\infty$, because $\int_{\Omega}d_\Omega(x)^{-\al-\be-sp}\,dx=\zeta_\Omega(\al+\be+sp)<\infty$.
       \end{proof}
       \begin{proof}[Proof of Theorem~\ref{tw1}, case II]
       Recall that in this case we assume that $\Omega$ is $(d-sp-\alpha-\beta)$-homogeneous.  Define the layers $\Omega_{i,n}=\left\{x\in\Omega:\frac{3}{2^{i+1}n}<d_{\Omega}(x)\leq\frac{3}{2^{i}n}\right\}.$ We observe that
       \begin{align*}
        n^{sp}\int_{\Omega_{\frac{3}{n}}}d_{\Omega}(x)^{-\alpha-\beta}\,dx&=n^{sp}\sum_{i=0}^{\infty}\int_{\Omega_{i,n}}d_{\Omega}(x)^{-\alpha-\beta}\,dx\\
        &\approx n^{sp+\alpha+\beta}\sum_{i=0}^{\infty}2^{-i(\alpha+\beta)}\left|\Omega_{i,n}\right|\\
        &\lesssim n^{sp+\alpha+\beta-\underline{d}}\sum_{i=0}^{\infty}2^{-i(\alpha+\beta+\underline{d})}=C,
       \end{align*}
       where $C$ is a constant independent of $n$. That means that the sequence $\{fv_{n}\}_{n\in\mathbb{N}}$ is bounded in $W^{s,p;\,\alpha,\beta}(\Omega)$. Now, the proof follows \cite[Proof of Theorem 2, case II]{DK}: we use Banach--Alaoglu and Eberlein--\v{S}mulian theorems to conclude that there exists a subsequence $\{fv_{n_{k}}\}_{k\in\mathbb{N}}$ convergent to $\ind_{\Omega}$ in $W^{s,p;\,\alpha,\beta}(\Omega)$.  The reflexivity of $W^{s,p;\,\alpha,\beta}(\Omega)$ is essential here.
       \end{proof}
  \begin{proof}[Proof of Theorem~\ref{tw1}, case III] 
       We proceed analogously as in the unweighted case in \cite[Proof of Theorem 2, case III]{DK}. In this case we just need to use the fractional weighted Hardy inequality (\ref{eq:fhi}) in the case (F) and Fatou's lemma to prove that the function $f=\ind_{\Omega}$ cannot be approximated by $C_{c}^{\infty}(\Omega)$ functions in $W^{s,p;\,\alpha,\beta}(\Omega)$. 
       \end{proof}
       \begin{rem}
       Notice that in the proof of the case III we use the fact that if $u_{n}\longrightarrow\ind_{\Omega}$ in $L^{p}(\Omega)$, then there exists a subsequence $u_{n_{k}}$ convergent to $\ind_{\Omega}$ almost everywhere; the same fact holds if we replace $L^{p}(\Omega)$ by the weighted space $L^{p}(\Omega,W)$ for almost everywhere positive $W\in L^{1}(\Omega)$.
       \end{rem}
       \begin{proof}[Proof of Theorem~\ref{tw2}]
       If $\int_{\Omega}|f(x)|^{p}d_{\Omega}(x)^{-sp-\alpha-\beta}\,dx<\infty$, then $f\in W_{0}^{s,p;\,\alpha,\beta}(\Omega)$, because by Lemma \ref{ineq} we have
       \begin{align*}
       [fv_{n}]^{p}_{W^{s,p;\,\alpha,\beta}(\Omega)}&\lesssim n^{sp}\int_{\Omega_{\frac{3}{n}}}\frac{|f(x)|^{p}}{d_{\Omega}(x)^{\alpha+\beta}}\,dx+\int_{\Omega_{\frac{3}{n}}}\int_{\Omega_{\frac{3}{n}}}\frac{|f(x)-f(y)|^{p}}{|x-y|^{d+sp}}d_{\Omega}(x)^{-\alpha}d_{\Omega}(y)^{-\beta}\,dy\,dx\\
       &\lesssim \int_{\Omega_{\frac{3}{n}}}\frac{|f(x)|^{p}}{d_{\Omega}(x)^{sp+\alpha+\beta}}\,dx+\int_{\Omega_{\frac{3}{n}}}\int_{\Omega_{\frac{3}{n}}}\frac{|f(x)-f(y)|^{p}}{|x-y|^{d+sp}}d_{\Omega}(x)^{-\alpha}d_{\Omega}(y)^{-\beta}\,dy\,dx\rightarrow 0,
       \end{align*}
       when $n\longrightarrow\infty$. On the other side, if $f\in W_{0}^{s,p;\,\alpha,\beta}(\Omega)$, then by the fractional Hardy inequality \eqref{eq:fhi} and Fatou's lemma we obtain that $\int_{\Omega}|f(x)|^{p}d_{\Omega}(x)^{-sp-\alpha-\beta}\,dx<\infty$. That proves the desired characterization of $W_{0}^{s,p;\,\alpha,\beta}(\Omega)$.
       \end{proof}
        \begin{proof}[Proof of Theorem~\ref{tw.emb}]
       This is a straightforward consequence of the fractional Hardy inequality \eqref{eq:fhi} in the case (T'), case I of the Theorem \ref{tw1} and Fatou's lemma. We can easily see that uniform domains are $\kappa$-plump, so \eqref{Hardy} is applicable.
       \end{proof}

\end{document}